\documentclass[11pt]{amsart}
\usepackage{hyperref}
\usepackage{color}
\usepackage{graphicx}
\usepackage{microtype}
\usepackage{amscd}
\usepackage{amssymb}

\newtheorem{thm}{Theorem}[section]
\newtheorem{lem}[thm]{Lemma}
\newtheorem{prop}[thm]{Proposition}
\newtheorem{cor}[thm]{Corollary}

\newcommand{\newtheorembox}[2]{\newtheorem{#1-boxless}[thm]{#2}\newenvironment{#1}{\begin{#1-boxless}}{\qed\end{#1-boxless}}}

\theoremstyle{definition}
\newtheorembox{defn}{Definition}
\newtheorembox{ex}{Example}
\newtheorembox{rmk}{Remark}
\newtheorembox{const}{Construction}

\newcommand{\an}{{\mathrm{an}}}

\newcommand{\CC}{\mathbb C}
\newcommand{\defi}[1]{\textbf{#1}}

\newcommand{\GL}{\mathit{GL}}
\newcommand{\Gm}{\mathbb G_m}

\newcommand{\QQ}{\mathbb Q}
\newcommand{\RR}{\mathbb R}

\newcommand{\ZZ}{\mathbb Z}

\title{Embeddings and immersions of tropical curves}

\author[Cartwright \and Dudzik \and Manjunath \and Yao]{Dustin Cartwright \and Andrew Dudzik \and Madhusudan Manjunath \and Yuan Yao}
\address{Department of Mathematics \\ University of Tennessee \\
         227 Ayres Hall \\ Knoxville, TN 37996}
\email{cartwright@utk.edu}
\address{Department of Mathematics \\ University of California \\
    970 Evans Hall \\ Berkeley, CA 94720}
\email{adudzik@math.berkeley.edu}
\address{Department of Mathematics \\ University of California \\
    970 Evans Hall \\ Berkeley, CA 94720
    }
\email{madhu@berkeley.edu}
\address{Department of Mathematics \\ University of Texas \\ 
    1 University Station C1200, Austin, TX 78712}
\email{yyao@math.utexas.edu}

\begin{document}

\begin{abstract}
We construct immersions of trivalent abstract tropical curves in the Euclidean plane and embeddings of all abstract tropical curves in higher dimensional Euclidean space. Since not all curves have an embedding in the plane, we define the tropical crossing number of an abstract tropical curve to be the minimum number of self-intersections, counted with multiplicity, over all its immersions in the plane. We show that the tropical crossing number is at most quadratic in the number of edges and
this bound is sharp.  For curves of genus up to two, we systematically compute the crossing number. Finally, we use our immersed tropical curves to
construct totally faithful nodal algebraic curves via lifting results of Mikhalkin and Shustin.
\end{abstract}

\maketitle

\section{Introduction}

Tropical geometry has been studied in two different flavors. One is the abstract
version that views a tropical variety as a skeleton of a Berkovich
analytification of an algebraic variety defined over a non-archimedean field.
The second is the embedded version that studies tropicalization of a variety
embedded in the algebraic torus or in projective space.  In this paper, we
examine the combinatorial relationship between these two different views.
First, we show that an abstract tropical curve can always be represented as an
embedded one:

\begin{thm}\label{thm:immersion-embedding}
Let $\Gamma$ be an abstract tropical curve and suppose that $d$ is the largest
degree of a vertex of~$\Gamma$. Then $\Gamma$ has a smooth embedding in $\RR^n$
when $n$ is at least $\max\{3, d-1\}$ and $\Gamma$ has an immersion in~$\RR^2$
if $d$ is at most~$3$.
\end{thm}

The conditions in Theorem~\ref{thm:immersion-embedding} are as sharp as possible
because the local model for a degree~$d$ vertex only embeds in~$\RR^{d-1}$, so
it is not possible to have an embedding or immersion in $\RR^n$ when $n < d-1$.
Moreover, while some abstract tropical curves can be embedded in~$\RR^2$, some
cannot, for example, because there might not even be an embedding of the
underlying graph. However, by Theorem~\ref{thm:immersion-embedding}, we always
have an immersion of an abstract tropical curve in~$\RR^2$ and we define the
\defi{tropical crossing number} to be the minimum number of crossings, counted
with tropical multiplicities, over all possible planar immersions. The tropical
crossing number is always bounded below by the graph-theoretic crossing number
of the underlying graph, which is the minimum number of crossings among all
planar immersions of the graph. Our proof of the existence of immersions also
bounds the tropical crossing number:

\begin{thm}\label{thm:immersion-bound}
If $\Gamma$ is a trivalent abstract tropical curve with $e$ edges, then the
tropical crossing number of~$\Gamma$ is at most $O(e^2)$.
\end{thm}

The quadratic bound in Theorem~\ref{thm:immersion-bound} is optimal up to a
constant factor, because the graph-theoretic crossing number of a random cubic
graph grows quadratically in the number of
edges~\cite[p.~143--144]{richter-salazar}. However, we show that the gap between
the two crossing numbers can itself grow quadratically, for an explicit family
of tropical curves:

\begin{thm}\label{thm:large-crossing-number}
There exists a family of trivalent abstract tropical curves with $e$ edges,
whose underlying graph is planar, but whose crossing number is $\Theta(e^2)$.
\end{thm}

The family in Theorem~\ref{thm:large-crossing-number} is the chain of loops with
bridges and generic edge lengths as in~\cite{jensen-payne}, where the bridges 
are necessary so that the graph is trivalent and therefore has a
finite crossing number. We use the result
from~\cite{cdpr} that these abstract tropical curves are Brill-Noether general,
or, more specifically, have the divisorial gonality of a Brill-Noether general
curve.
Any such family of abstract tropical curves would also have quadratic tropical crossing number.

Having looked at the asymptotic behavior of tropical crossing numbers for large
graphs, we then turn to the small graphs, specifically of genus at most~$2$.
Perhaps surprisingly, in Proposition~\ref{prop:genus-2}, we give a case of an
abstract tropical curve which has a planar embedding for generic values of the
metric parameters, but not for specializations. As a consequence, the crossing
number is neither lower nor upper semi-continuous in the metric parameters (see
Remark~\ref{rmk:semicontinuity}).

A complementary analysis of curves with tropical crossing number~$0$, going up
to genus~$5$ has been undertaken in~\cite{bjms}. They describe the closure of
the locus of curves with crossing number~$0$ inside the moduli space of all
curves, in terms of equalities and inequalities on the metric parameters of any
given graph.
Their
techniques are computational and demonstrate that there are effective and
practical algorithms for classifying abstract tropical curves with a given crossing number.

Using the realizability results of Mikhalkin and Shustin, we have the following
application to algebraic curves and their analytifications.
\begin{thm}\label{thm:faithful-immersion}
Let $K$ denote the field of convergent Puiseux series with complex coefficients.
If $\Gamma$ is a trivalent abstract tropical curve with rational edge lengths,
then there exists a $K$-curve~$C$, such that the minimal skeleton of the
Berkovich analytification $C^\an$ is isometric to~$\Gamma$, and there is an
immersion from an open subset $C' \subset C$ to $\Gm^2$ with totally faithful
tropicalization.
\end{thm}
\noindent The field of convergent Puiseux series in
Theorem~\ref{thm:faithful-immersion} refers to the subfield of the field of
formal Puiseux series $\CC\{\!\{t\}\!\}$ consisting of those series which
converge for sufficiently small values of~$t$. Thus, the $K$-curve from
Theorem~\ref{thm:faithful-immersion} can be explicitly specialized to a curve
over~$\CC$ by choosing a sufficiently small value of~$t$.

The totally faithful tropicalization appearing in
Theorem~\ref{thm:faithful-immersion} is a strengthening of the faithful
tropicalizations studied in~\cite{bpr} and~\cite{grw}, but for nodal curves
rather than smoothly embedded curves.  Baker, Payne, and Rabinoff proved 
that for any algebraic curve~$C$ and a skeleton of its analytification~$C^\an$,
there exists an embedding of an open dense subset $C \supset C' \rightarrow
\Gm^n$, such that the projection of $(C')^\an$ to its tropicalization is an
isometry on its skeleton~\cite[Thm.~1.1]{bpr}. Totally faithful tropicalizations
were defined
in~\cite{log-deformation} as the strengthening where the
tropicalization is required to be an isometry also on the skeleton of~$C'$,
which contains additional unbounded edges for every point in $C \setminus C'$.
We work with the nodal version of a totally faithful tropicalization, where $C'
\rightarrow \Gm^2$ is only an immersion, and the tropicalization correspondingly
also has nodes. The analogue of Theorem~\ref{thm:faithful-immersion}
for embeddings in higher dimensions appears in~\cite{log-deformation}, and
combines our results on embeddings of tropical curves with their work on
realizations of such curves.

The rest of the paper is organized as follows. In
Section~\ref{sec:embeddings-immersions} we prove
Theorems~\ref{thm:immersion-embedding} and~\ref{thm:immersion-bound} on the
existence of embeddings and immersions of tropical curves.
Section~\ref{sec:gonality} proves Theorem~\ref{thm:large-crossing-number} by
relating the crossing number to the divisorial gonality.
Section~\ref{sec:classlowgenus} contains a study of crossing numbers for graphs
of genus at most~$2$, and Section~\ref{sec:realizations} proves
Theorem~\ref{thm:faithful-immersion}, applying our results to realizability
questions for algebraic curves.

\subsection*{Acknowledgments}
Our work was initiated in the Mathematics Research Communities (MRC) program on
``Tropical and Non-Archimedean Geometry'' held in Snowbird, Utah in Summer 2013.
We thank the MRC as well as the organizers of our program, Matt Baker and Sam
Payne, for their support and guidance during and after the program.
We would also like to thank Mandy Cheung,
Lorenzo Fantini, Jennifer Park, and Martin Ulirsch for sharing their results
from the same workshop on log deformations of curves~\cite{log-deformation}, which helped to motivate
our work.  We would also like to acknowledge Melody Chan and Bernd Sturmfels for several interesting discussions.

Madhusudan Manjunath was supported by a Feoder-Lynen Fellowship of the Humboldt Foundation and an AMS-Simons Travel Grant during this work.

\section{Embeddings and immersions of tropical
curves}\label{sec:embeddings-immersions}

In this section, we present the construction of the smooth planar immersion and
embedding of tropical curves in $\mathbb{R}^n$.  We start by recalling the
definitions of abstract and embedded tropical curves, as well as the relation
between the two.

\begin{defn}
An \defi{abstract tropical curve} is a finite connected graph possibly with
loops or multiple edges together with
either a positive real number or infinity attached to each edge, which will be
known as the \defi{length} of the edge. Any edge with infinite length must have
a degree~$1$ vertex at one of its endpoints, which will be referred to as an
\defi{infinite vertex}.

A \defi{subdivision} of an abstract tropical curve consists of replacing an edge
with two consecutive edges whose lengths add up to the length of the original
edge. If the original edge was infinite then the subdivided edge incident to the
infinite vertex must also be infinite.
Two tropical curves are \defi{equivalent} if one can be
transformed into the other by a series of subdivisions and reverse subdivisions.
\end{defn}

Our definition of an abstract tropical curve is based on the one
in~\cite{mikhalkin}, but slightly more general because we do not require all
$1$-valent vertices to be infinite. Because of this, our definition is
equivalent to that in~\cite[Sec.~2.1]{abbr}.

\begin{rmk}\label{rmk:curve-metric-space}
The underlying graph of a tropical curve has a natural realization as a
topological space and the lengths along the edges additionally give a metric on
this realization, away from the infinite vertices. This metric realization gives
an alternative characterization of abstract tropical curves up to equivalence as inner
metric spaces which have a finite cover by open sets isometric to star shapes.
See, for example, \cite[Sec. 2.1]{abbr}, for a definition from this perspective.
\end{rmk}

We will also consider a coarsening of the above equivalence for tropical modification. An \defi{elementary tropical modification} is formed
by adding an infinite edge at a finite vertex. A \defi{tropical modification} is any sequence of elementary tropical modifications,
subdivisions, and reverse subdivisions.

\begin{figure}
\includegraphics{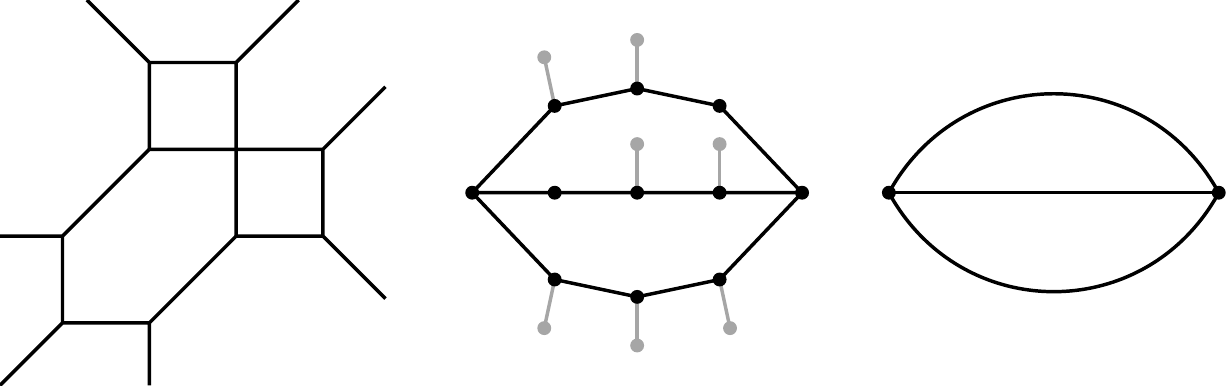}
\caption{The immersed tropical curve on the left gives rise to the abstract
tropical curve in the center. The vertices and edges shown in gray are infinite
and the other edges all have finite lengths, which happen to be equal. This
abstract tropical curve is a tropical modification of the curve on the right,
in which all edge lengths are equal to $4$ times the finite edge lengths of its
tropical modification. We will refer to the
underlying graph of this last curve as a theta graph.}
\label{fig:theta-example}
\end{figure}

\begin{ex}\label{ex:theta-modification}
In the center of Figure~\ref{fig:theta-example} is an abstract tropical curve,
which is a tropical modification of the tropical curve on
the right of that figure. This tropical modification is obtained by first
subdividing each of the edges of the latter graph into 4 edges of equal length
and then attaching infinite edges to 7 of the 9 newly created vertices.
In Example~\ref{ex:theta-immersion}, we will see that the leftmost diagram in
Figure~\ref{fig:theta-example} gives an immersion of the modified graph.
\end{ex}

We now turn to the embedded side and define smooth and nodal tropical curves
in~$\RR^n$. For any integer~$d$ in the range $2 \leq d \leq n+1$, the standard
smooth model of valence~$d$ is the union of the $d$ rays generated by the
coordinate vectors $e_1$ through $e_{d-1}$ and the vector $-e_1-\cdots
- e_{d-1}$. Up to changes of coordinates in $\GL_n(\mathbb{Z})$ these are exactly the $1$-dimensional
matroidal fans from~\cite{ardila-klivans} and such fans form the building blocks
of tropical manifolds~\cite[Def.~1.14]{mikhalkin-zharkov}.

\begin{defn}\label{def:smooth}
A \defi{smooth tropical curve} in~$\RR^n$ will be a union of finitely many
segments and rays such that a neighborhood of any point is equal to a
neighborhood of a standard model, after a translation and a change of
coordinates taken from $\GL_n(\ZZ)$.
\end{defn}

For nodal curves, we have an additional local model. Recall that in algebraic
geometry, a nodal curve singularity is one that is analytically isomorphic to a
union of two distinct lines meeting at a point. For tropical curves, the local model for
a node will analogously consist of two distinct (classical) lines with rational slopes,
passing through the origin in~$\RR^2$. We will only ever consider nodal curves
in the plane.

\begin{defn}\label{def:nodal}
A \defi{nodal tropical curve} in~$\RR^2$ is a union of finitely many segments and rays which is
locally equal to either to a standard smooth local model or a nodal local model, again up
to translation and the action of~$\GL_2(\ZZ)$.
\end{defn}

The embedded and abstract tropical curves are related in that any embedded curve
gives rise to an abstract one, as we now explain. By definition, each smooth
tropical curve is a union of segments and rays, so if we add a vertex ``at
infinity'' for each unbounded ray, we naturally get a finite graph. Thus, it
only remains to assign lengths to the edges of this graph. In the standard local model, since we only allow
changes of coordinates in $\GL_n(\ZZ)$, any segment is parallel
to an integer vector.
In other words, if $p_1$ and~$p_2$ are the endpoints of a segment, then
$p_1 - p_2 = \alpha v$, where $\alpha$ is a positive real number and $v$ is a
primitive integral vector i.e., $v \in
\ZZ^n$ and the entries have no common divisor. 
Then $\alpha$ is uniquely determined and we use it as the
length of the edge. This is the same metric used in Mikhalkin's enumerative
results~\cite[Rmk.~2.4]{mikhalkin}.

\begin{figure}
\includegraphics{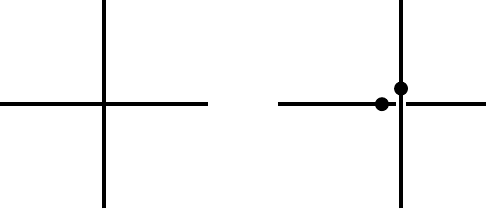}
\caption{Procedure for resolving the nodal singularities in a plane curve. On
the left is a neighborhood of a node in an embedded tropical curve and on the
right, its resolution as an abstract tropical curve, which includes two vertices
and the ends of four edges.}
\label{fig:resolution}
\end{figure}

For nodal curves in~$\RR^2$, we use the procedure as above for edges and smooth
vertices. For the local model consisting of two lines passing through the
origin, the abstract tropical curve has two vertices which map to the origin,
one of which is an endpoint for the edges corresponding to one line and the other corresponding to the
other line. This is illustrated in Figure~\ref{fig:resolution}.

\begin{defn}
Let $\Gamma$ be an abstract tropical curve. An \defi{embedding} in~$\RR^n$
(resp.\ an \defi{immersion} in~$\RR^2$)
is a smooth (resp.\ nodal) tropical curve in $\RR^n$ (resp.\ $\RR^2$) which is isometric to a tropical
modification of~$\Gamma$.
\end{defn}

\begin{ex}\label{ex:theta-immersion}
The leftmost curve of Figure~\ref{fig:theta-example} is a nodal tropical curve
with a single node. After resolving the node, we obtain the abstract tropical
curve in the center, so the leftmost curve is an immersion of the center curve. Although the segments in the immersed curve have different
lengths in the Euclidean metric, each finite edge of the realized graph has the
same length.

Moreover, as we saw in Example~\ref{ex:theta-modification}, the center abstract
tropical curve is a tropical modification of the curve on the right, and thus
the leftmost curve is also an immersion of the rightmost abstract tropical
curve.
\end{ex}

We define the notion of multiplicities and then present a proof of
Theorems~\ref{thm:immersion-embedding} and~\ref{thm:immersion-bound}. Recall from Definition~\ref{def:nodal} that a neighborhood of
a nodal point of a tropical curve is the union of two lines. As in the
construction of the edge lengths, we can assume that these lines are parallel to
integral vectors $v = (v_1, v_2)$ and $u = (u_1, u_2)$ respectively, and that
these integral vectors are primitive. Then the multiplicity of this node is
\begin{equation}\label{eq:node-multiplicity}
\left\lvert \det \begin{pmatrix} v_1 & u_1 \\ v_2 & u_2
\end{pmatrix}\right\rvert
\end{equation}
This determinant is equivalent to the multiplicity of the stable intersection of
the two lines, such as in~\cite[Def.~3.6.5]{maclagan-sturmfels}.
We now prove our main theorems of this section.

\begin{proof}[Proof of Theorems~\ref{thm:immersion-embedding}
and~\ref{thm:immersion-bound}]
Fix $n \geq 2$ and let $\Gamma$ be a graph whose vertices have degree at most
$n+1$. We will construct an embedding (if $n \geq 3$) or an immersion (if $n =
2$) of~$\Gamma$, as follows.

If $\Gamma$ has any loops or parallel edges, we first subdivide them to obtain an equivalent graph
which does not have loops or parallel edges. We then label each vertex-edge incidence of~$\Gamma$
with an integer from $0$ to~$n$ inclusive such that the labels are distinct at
each vertex, and along each edge, the labels at the endpoints differ by~$\pm 1$,
modulo $n+1$. The first part of this condition can always be
achieved since the degree of any vertex is at most $n+1$ and if the second part
is violated at an edge, then we can resolve it by subdividing that edge
sufficiently many times.

Fix a sufficiently small real
number~$L$ such that $nL$ is less than the length of any edge of~$\Gamma$ and we
will construct most of the embedding inside the $n$-dimensional cube~$[0,L]^n$.
We begin by defining $\iota\colon \Gamma \rightarrow \RR^n$ on the finite
vertices of~$\Gamma$ by sending them to points inside this $n$-dimensional cube,
and we will later assume that these points are generic. 

To extend $\iota$ to a small neighborhood of each vertex, we let
$e_1,\dots,e_n$ denote the basis elements $(1,0,\dots,0)$, $(0,1,\dots,0)$,
\dots, $(0,0,\dots,1)$ of $\mathbb{Z}^n$ and set $e_0 = (-1, \ldots, -1)$.
In a neighborhood of a vertex~$v$, we send a small interval along each edge
incident to~$v$ to an interval in the direction~$e_i$ if $i$ is the label on the
edge-vertex pair. In order to complete the local model, we add infinite edges
at~$v$ until its degree is $n+1$ and then send these infinite edges to rays in
the direction $e_i$ for the indices $i$ which were not among the labels
around~$v$. We also send the infinite edges of~$\Gamma$ to rays in the already
determined directions.

We now extend the map~$\iota$ to the finite edges of~$\Gamma$ as follows.
For each edge~$E$ between vertices $u$ and~$w$ in~$\Gamma$, we already defined
$\iota$ on neighborhoods of each endpoint, thus giving the initial
tangent directions.
We adopt the convention that the index~$i$ for the vector~$e_i$ is taken
modulo $n+1$, so $e_{n+1} = e_0$, and so on.
Thus, we can assume that the tangent directions at~$u$ and~$w$ are
$e_i$ and~$e_{i+1}$ respectively.
Since $e_{i}, \ldots, e_{i+n-1}$ form a basis for $\RR^n$, we can
uniquely write:
\begin{equation*}
\iota(w) - \iota(u) = \alpha_0 e_i + \alpha_1 e_{i+1} + \cdots + \alpha_{n-1}
e_{i + n-1}
\end{equation*}
We set $m = \lvert \alpha_0\rvert + \cdots + \lvert \alpha_{n-1} \rvert$, and
let $\ell$ be the length of the edge~$E$.
Recall that $\iota(w)$ and~$\iota(u)$ are inside a box of side
length~$L$, so $m \leq nL < \ell$, by our choice of~$L$.

We now let $\alpha_0'$, $\alpha_0''$, $\alpha_1'$, and $\alpha_1''$ be the
unique numbers such that for $i = 0, 1$, we have $\alpha_i' > 0$, $\alpha_i'' <
0$, $\alpha_i' + \alpha''_i = \alpha_i$, and
\begin{equation*}
\lvert \alpha_i'
\rvert + \lvert \alpha_i'' \rvert - \lvert \alpha_i \rvert = (\ell - m)/2.
\end{equation*}
Then, we can map the edge~$E$ to connect $\iota(u)$ and~$\iota(w)$ in a
piecewise linear fashion, with all segments parallel to one of $e_i, \ldots,
e_{i+n-1}$. In particular, we define $\iota(E)$ to linearly interpolate between
the following points:
\begin{align*}
\iota(u)& \\
\iota(u) & + \alpha_0' e_i \\
\iota(u) & + \alpha_0' e_i + \alpha_1' e_{i+1} \\
\iota(u) & + \alpha_0' e_i + \alpha_1' e_{i+1} + \alpha_0'' e_i \\
\iota(u) & + \alpha_0' e_i + \alpha_1' e_{i+1} + \alpha_0'' e_i
  + \alpha_{n-1} e_{i+n - 1} \\
\iota(u) & + \alpha_0' e_i + \alpha_1' e_{i+1} + \alpha_0'' e_i
  + \alpha_{n-1} e_{i+n - 1} + \alpha_{n-2} e_{i+n-2} \\
& \vdots \\
\iota(u) & + \alpha_0' e_i + \alpha_1' e_{i+1} + \alpha_0'' e_i
  + \alpha_{n-1} e_{i+n-1} + \alpha_{n-2} e_{i+n-2} + \cdots
  + \alpha_2 e_{i+2} \\
& \qquad = \iota(w) - \alpha_1'' e_{i+1} \\
\iota(w)& 
\end{align*}
The constants $\alpha_i'$ and $\alpha_i''$ are chosen such that this path has
total length~$\ell$, which is the desired edge length, and such that the tangent
directions from $\iota(u)$ and $\iota(w)$ are $e_i$ and~$e_{i+1}$, respectively.
In order to complete the local model at each vertex other than the endpoints
$\iota(u)$ and~$\iota(v)$, we need to add an unbounded ray in the appropriate
direction, corresponding to a subdivision and modification of~$\Gamma$, which we
will call $\Gamma'$.

To show that $\iota$ defines an immersion for $n=2$, we need to show that for
generic choices of images for the vertices of~$\Gamma$, the self-intersections
of~$\iota$ will only occur at edges. If $w$ is a vertex of $\Gamma$, then a
small perturbation of $\iota(w)$ will move $\iota(w)$ off of any other point of
$\iota(\Gamma')$, even if that point is in an edge of $\Gamma$ containing $w$.
On the other hand, if $v$ is a vertex of~$\Gamma'$ formed by the subdivision
of~$\Gamma$, then the key observation is that both coordinates of $\iota(v)$
will be affected by at least one of the endpoints of the original edge~$e$ on
which $v$ occurs. Moreover, even if $\iota(v)$ is contained in an edge
of~$\Gamma$ sharing an endpoint with~$e$, then the two edges start in different
directions from their endpoints, so, in $\Gamma'$, the vertex $v$ is not
adjacent to the common endpoint, so we can perturb $\iota(v)$ using the other
endpoint of~$e$. Since $\Gamma$ has no parallel edges, this means that for any
intersection of $\iota(v)$ with another edge, a small perturbation of one of the
vertices of $\Gamma$ will push $\iota(v)$ off of that edge without affecting the
edge. By keeping the perturbation sufficiently small, we will not introduce any new
self-intersections, and thus we can obtain the desired immersion.

For $n\geq 3$, we will show that for generic choices for an embedding of the vertices, $\iota$ is injective.
Here, assume that we have an intersection between two different embedded edges.
By our initial subdivisions, these edges share at most one endpoint in common.
Moreover, by our choice of the directions for the edge near the endpoints, even
if the edges share an endpoint~$v$, the intersection will not be in the first two
segments next to~$v$. Thus, at least two of the coordinates of this
point of intersection will depend on the opposite endpoint, so there is at least
a $2$-dimensional space of perturbations that will move the segment. On the
other hand, the perturbations parallel to the other segment will still have a
self-intersection, but this will be at most a $1$-dimensional subspace. Thus, we
can find some perturbation which eliminates the intersection.

Finally, we need to prove the quadratic bound on the number of crossings in the
case of immersions in $\mathbb{R}^2$. The subdivisions to avoid loops or parallel edges can be
done be replacing each edge by $2$ or~$3$ edges respectively. Then, our
immersion further subdivides each edge into $4$ segments and introduces $3$
infinite rays. Thus, each edge of~$\Gamma$ results in a bounded number of
segments in the immersion. Moreover, these segments are each parallel to one of
the vectors $e_0$, $e_1$, and~$e_2$, so the intersection multiplicity of any two
is at most one. Thus, the total number of crossings of~$\iota$ is $O(e^2)$, as
desired.
\end{proof}

\begin{rmk}\label{rmk:immersion-embedding-extras}
We note that the embeddings and immersions constructed in the proof of
Theorem~\ref{thm:immersion-embedding} have the following additional properties,
which will be relevant for the applications to realizability of curves
in Theorem~\ref{thm:faithful-immersion} and in~\cite{log-deformation}. First, for every edge of~$\Gamma$, the directions of
the embeddings of the edges in the subdivision form a basis for $\ZZ^n$. Second,
if all of the lengths of~$\Gamma$ are rational, then all the vertices of the
tropical curve can also be chosen to be rational. Third, vertices contained in
three or more bounded edges are only adjacent to vertices contained in at most
two bounded edges, because every edge of the original graph is subdivided.
\end{rmk}

\begin{rmk}\label{rmk:generic-projection}
In algebraic geometry and other areas, immersions and embeddings may
be constructed by starting with an embedding in a high-dimensional space and
projecting (for example, Prop.~IV.3.5 and Thm.~IV.3.10 in~\cite{hartshorne}).
However, the key to such arguments is showing that generic projections preserve
embeddings or create immersions. However,
in tropical geometry, projections which preserve smoothness,
even at a single point are relatively rare and thus can not be considered to be
the generic case.
\end{rmk}

\section{Crossing numbers and gonality}\label{sec:gonality}

By Theorem~\ref{thm:immersion-embedding}, any trivalent tropical curve has a planar
immersion. We can define the \defi{tropical crossing number} of the curve to
be the minimal number of nodes, counted with multiplicity, of any planar
immersion. Thus, the tropical curve has a smooth planar embedding if and
only if its crossing number is zero, and Theorem~\ref{thm:immersion-bound}
establishes a quadratic upper bound on the crossing number. In this section, we
establish a lower bound on the crossing number of an abstract tropical curve in
terms of its genus and divisorial gonality. We use this to prove
Theorem~\ref{thm:large-crossing-number} in the form of
Corollary~\ref{cor:chain-loops}.

A tool we will repeatedly use in this section and the next is the dual
subdivision of a plane curve. We gather some basic facts about the dual
subdivision which will come up in many of the proofs. Since any nodal tropical
curve~$\Gamma$ is balanced, \cite[Thm.~3.3]{rgst} states that it is the
non-differentiable locus of a concave piecewise-linear function, whose slopes
are integral and such that the difference between the slope vectors on either
side of an edge of $\Gamma$ has relatively prime entries. The dual
subdivision~$\Delta$ is the projection of the lower convex hull of the slopes of
the piecewise linear function. The slopes on either side of an edge of $\Gamma$
will be joined by an edge of this subdivision and since the difference vector
has relatively prime entries, the only integral points contained in the edge are
its endpoints. See~\cite[Sec.~1.3]{maclagan-sturmfels} for details on the dual
subdivision.

For us, the relevance of the above construction will be the duality between the
curve and the subdivision, in which the vertices and edges of the curve
correspond to the polygons and edges of the subdivision respectively. Likewise,
the bounded and unbounded regions of the complement of the curve correspond to
the vertices in the interior and on the boundary of~$\Delta$. An example of a
tropical curve and its dual subdivision are shown in
Figure~\ref{fig:subdivision}.

\begin{figure}
\includegraphics{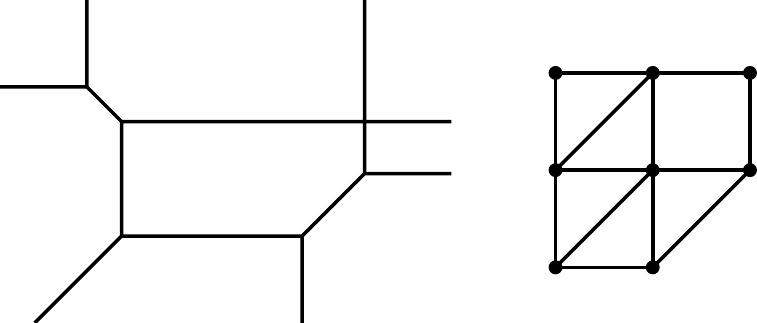}
\caption{A plane tropical curve on the left and the corresponding dual
subdivision on the right. The triangles of the subdivision correspond to the
trivalent vertices of the curve and the square represents the curve's unique
node.} \label{fig:subdivision}
\end{figure}

\begin{defn}
The \defi{genus}~$g$ of an abstract tropical curve is the rank of the first
homology of its underlying graph, i.e.\ $\dim_\QQ H_1(\Gamma, \QQ)$.
\end{defn}

\begin{prop}\label{prop:interior-points}
If $\Gamma$ is a tropical curve of genus~$g$ and $\iota(\Gamma)$ is an immersion
with $n$ nodes, counted with multiplicities, then
\begin{equation*}
i = g + n,
\end{equation*}
where $i$ is the number of integral points in the interior of the dual polygon.
\end{prop}

We will use the following well-known result about the number of lattice points
in an integral polygon.

\begin{prop}[Pick's Theorem]\label{prop:pick}
Let $P$ be a polygon with integral vertices. If $P$ has area~$A$ and $i$ lattice
points, of which $b$ are on the boundary of the polygon, then we have the
relation $A = i - b/2 - 1$.
\end{prop}

\begin{proof}[Proof of Prop.~\ref{prop:interior-points}]
We first introduce some additional notation. We write $g'$ for the genus
of~$\iota(\Gamma)$, without resolving any nodes, and $n'$ for the number of such
nodes, without any multiplicities. Since the procedure for resolving a node
adds a vertex, but does not change the number of edges, we have that $g = g' -
n'$. Therefore, it will suffice to prove the relation
\begin{equation}\label{eq:genus-modified}
i = g' + n - n'.
\end{equation}

The dual subdivision~$\Delta$ will have a triangle corresponding to each
trivalent vertex and a parallelogram corresponding to each node. We now count
how the interior lattice points of~$\Delta$ fall within the subdivision. Because
of our local smooth model, the only lattice points in a triangle will be its
vertices.
Moreover, as noted above, there are no
lattice points on the interior of any edge of the subdivision. Thus, an internal
lattice point of~$\Delta$ is either a vertex of the subdivision or in the
interior of a parallelogram. Each vertex of the former category defines a
bounded region of the complement of $\iota(\Gamma)$ and thus there are $g'$ such
vertices.
Therefore, we will have verified~(\ref{eq:genus-modified}) if we
can show that a node of multiplicity $m$ is dual to a parallelogram containing
$m-1$ interior lattice points.

Let $P$ be such a parallelogram. A simple computation shows that the area of~$P$
is equal to the multiplicity~$m$ of the node as in~(\ref{eq:node-multiplicity}). As noted
above, the only lattice points on the boundary of~$P$ are the $4$ vertices.
Therefore, Pick's formula (Prop.~\ref{prop:pick}) tells us that $P$ contains
$m+3$ lattice points and thus $m-1$ interior lattice points.
\end{proof}

Our method for proving the asymptotic lower bound on the crossing number from
Theorem~\ref{thm:large-crossing-number} is the divisorial gonality of graph as
in~\cite{baker}. Specifically a graph~$\Gamma$ is defined to have
\defi{divisorial gonality}~$d$ if $d$ is the least integer such that there
exists a divisor on~$\Gamma$ of degree~$d$ and rank~$1$. We refer
to~\cite[Sec.~1]{gathmann-kerber} or~\cite[Def.~7.1]{mikhalkin-zharkov08} for
the definitions of the rank of a divisor and
to~\cite[Sec.~3.6]{maclagan-sturmfels} for the definition of stable
intersection. If $D$ is a divisor on~$\Gamma$, meaning a formal sum of points,
then we define $\iota_*(D)$ to be the formal sum of points in~$\RR^2$ by
taking the images of the points of~$D$.
This construction relates the divisor theory of the abstract tropical curve with
stable intersections of the immersion by the following lemma.

\begin{lem}\label{lem:stable}
Let $a$ be a real number and let $L$ be the line in $\RR^2$ defined by $x = a$.
If $\phi \colon \RR^2 \rightarrow \RR$ is the piecewise
linear function $\phi(x) = \max\{x - a, 0\}$, then $\phi \circ \iota$ is a
piecewise linear function with integer slopes, so that it has a divisor
$\operatorname{div}(\phi \circ \iota)$, such that
$\iota_*(\operatorname{div}(\phi\circ\iota))$ equals the stable intersection of
$L$ with $\iota(\Gamma)$.
\end{lem}

\begin{proof}
By~\cite[Def.~3.6.5]{maclagan-sturmfels}, the stable intersection consists of
those points contained both in $L$ and in a segment of $\iota(\Gamma)$ which is
not parallel to~$L$. Let $p$ be such a point and we compute the multiplicity of
the stable intersection at~$p$ using a small perturbation of~$L$ to the right so
that the multiplicity is a summation over segments of $\iota(\Gamma)$ containing
both $p$ and also extending to the right of~$L$. Fix one such segment and let $v
= (v_1, v_2)$ be the minimal integral vector parallel to it. Then, the
contribution of this segment to the multiplicity is the index of the lattice
generated by $(0,1)$ and $v$ in $\ZZ^2$ by \cite[Def.~3.6.5]{maclagan-sturmfels}. The former lattice can also be
generated by $(0,1)$ and $(e_1, 0)$, so the index is $e_1$.

On the other hand, $\phi \circ \iota$ is non-zero to the right of~$L$ and so its
divisor can also be computed by taking the slopes along the segments which
extend to the right of~$L$. If $e$ is again the minimal integral vector parallel
to such a segment, then, by our definition of edge lengths, the slope is $\phi(p
+ e) - \phi(p) = e_1$, which is in integer. Thus, the two multiplicity
computations coincide, and so we have the desired equality from the lemma
statement.
\end{proof}

\begin{prop}\label{prop:hyperplane-section}
Let $\iota\colon \Gamma\rightarrow \RR^2$ be an immersion of a tropical curve in
the plane and suppose that $\iota(\Gamma)$ is not just a line. Then the stable
intersection of~$\iota(\Gamma)$ with a straight line of rational slope defines a
divisor with rank at least~$1$.
\end{prop}

\begin{proof}
We first change coordinates so that the line~$L$ is parallel to the $y$-axis,
say defined by $x = a$. Let $D$ be the divisor formed by the stable intersection
of~$\Gamma$ with~$L$. Then we need to show that for any point~$p \in \Gamma$,
there exists an effective divisor linearly equivalent to~$D$ which contains~$p$.
We first let $L'$ be the vertical line containing $\iota(p)$, and say that $L'$
is defined by $x = a'$.
Let $D'$ the stable intersection of~$L'$ with~$\iota(\Gamma)$. We define $\phi\colon
\RR^2 \rightarrow \RR$ to be the piecewise linear function
\begin{equation*}
\max\{x-a', 0\} - \max\{x-a, 0\}.
\end{equation*}
Then, using Lemma~\ref{lem:stable} and the linearity the $\operatorname{div}$ function, $D' = D + \operatorname{div}(\phi\circ\iota)$, so $D'$ is linearly
equivalent to~$D$. Moreover, if $p$ is an isolated point of $L' \cap
\iota(\Gamma)$, then $p$ is in $D'$ as required.

Otherwise, $p$ is contained in a positive-length interval of $L' \cap
\iota(\Gamma)$. If this interval is bounded, then, by the local model of a
smooth curve, both its endpoints will be
contained in the stable intersection $D'$.
These endpoints can be moved together the same distance so
that one of them contains~$p$.
If the interval is unbounded, then by our assumption,
it must still be bounded in one direction, and $D'$ will contain that endpoint.
We can then move the point of $D'$ along the unbounded edge until it contains $p$.
\end{proof}

\begin{prop}\label{prop:gonality}
If $\Gamma$ is an abstract tropical curve with divisorial gonality~$d > 2$ and genus~$g$,
then the tropical crossing number of~$\Gamma$ is at least $\frac{3}{8}(d-2)^2
-g + \frac{1}{2}$.
\end{prop}

Our proof uses the same technique as in~\cite[Thm.~3.3]{smith}, which bounds the
gonality of a curve in a smooth toric variety.

\begin{proof}
Suppose that $\Gamma$ has divisorial gonality~$d$ and we have a planar
immersion. By~\cite[Thm. 3.3]{rgst}, $\Gamma$ is dual to a subdivision of a
Newton polygon~$\Delta$. Recall that the \defi{lattice width} $w$ of the polygon~$\Delta$ is
the smallest possible value of $\max\{\lambda \cdot \mathbf x \mid \mathbf x \in
\Delta\} - \min\{\lambda
\cdot \mathbf x \mid \mathbf x \in \Delta\}$ as $\lambda \in \ZZ^2$ ranges over
all non-zero integer vectors. Then, applying
Proposition~\ref{prop:hyperplane-section} to the line defined by $\lambda \cdot
\mathbf x = 0$, we see that
the divisorial gonality
of~$\Gamma$ is at most~$w$, i.e.\ $w \geq d$. We let $\Delta^{(1)}$ denote the
convex hull of the interior points of~$\Delta$ and let $w'$ denote the lattice
width of~$\Delta^{(1)}$. Then Theorem~4 from~\cite{castryck-cools} shows that
$w' = w-2$ or $\Delta$ and~$\Delta^{(1)}$ are unimodular simplices scaled by~$w$
and $w-3$ respectively. We first deal with the former case, for which we have $w' \geq d-2$.

Then, Theorem~2 from~\cite{toth-makai} implies that $A$, the area
of~$\Delta^{(1)}$ is at least $3(d-2)^2 / 8$.
By our assumption that $d > 2$, we know that $w' > 0$, so $\Delta^{(1)}$ is not
contained in a line. We thus apply Pick's formula,
Proposition~\ref{prop:pick}, to get the
relation $i = A + b/2 + 1$, where
$i$ is the number of integral points in~$\Delta^{(1)}$, and $b$ is the number of
those integral points on the boundary of~$\Delta^{(1)}$. Thus,
\begin{equation*}
i \geq A + 1 \geq \frac{3(d-2)^2}{8} + 1.
\end{equation*}

We now return to the case when $\Delta^{(1)}$ is a unimodular simplex scaled by
$w-3$. We can directly compute that in this case, with $i$
again the number of integral points in~$\Delta^{(1)}$,
\begin{equation*}
i = \frac{(w-2)(w-1)}{2} \geq \frac{(d-2)(d-1)}{2} =
\frac{(d-2)^2}{2} + \frac{d-2}{2}
\geq \frac{3(d-2)^2}{8} + \frac{1}{2}.
\end{equation*}
Thus, in either case, we can apply 
Proposition~\ref{prop:interior-points} to get that the number
of nodes is $i - g$ and thus at least $\frac{3}{8}(d-2)^2 - g + \frac{1}{2}$,
as claimed.
\end{proof}

We prove Theorem~\ref{thm:large-crossing-number} using the chain of loops with
bridges from~\cite{jensen-payne}, which is similar to the chain of loops
from~\cite{cdpr}, but with bridges added between the loops. In particular, these
graphs are planar. As in~\cite{cdpr} and~\cite{jensen-payne}, we will assume
that the lengths of the edges in the loops are generic, for which it is
sufficient for the vector of length assignments for edges in the loops to avoid
a finite union of rational hyperplanes. The following then shows that this
family is far from having planar embeddings for large~$g$.

\begin{cor}\label{cor:chain-loops}
If $\Gamma$ is the chain of $g \geq 3$~loops with bridges and has generic edge lengths then the
crossing number of~$\Gamma$ is at least $3 g^2 / 32 - 11 g/ 8 + 7/8$. In
particular, the crossing number of this family of graphs is quadratic in the
number of edges.
\end{cor}

\begin{proof}
By~\cite[Thm.~1.1]{cdpr} and the fact that divisorial gonality does not change
by adding bridges~\cite[Rmk.~1.2]{jensen-payne}, we deduce that $\Gamma$ is Brill-Noether general.
This implies that it has divisorial gonality $\lceil g/ 2 \rceil + 1$. In particular, when $g \geq
3$, the divisorial gonality is greater than $2$. Substituting this into the
expression in Proposition~\ref{prop:gonality} and dropping the ceiling, we see
that $\Gamma$ has crossing number at least $3 g^2 / 32 - 11 g / 8 + 7/8$.

The last sentence follows from the fact that $\Gamma$ has $3g-3$ edges and thus
the crossing number is also quadratic in the number of edges.
\end{proof}

\section{Low genus curves}{\label{sec:classlowgenus}}

In this section, we examine the crossing numbers of tropical curves with
genus at most~$2$. We begin with genus~$0$. 
The underlying graph of a genus~$0$ tropical curve is a tree and in
Proposition~\ref{prop:tree}, we characterize genus~$0$ tropical curves with
crossing number~$0$ in terms of its underlying graph. In genus~$1$, we only
consider graphs consisting of infinite edges attached to the central loop, in
which case Proposition~\ref{prop:sun} gives a sharp bound on the crossing
number. For genus~$2$ curves, we restrict to \defi{stable tropical curves}, by
which we mean that every vertex
has degree equal to~$3$. In particular, stable tropical curves have no infinite
edges.
Proposition~\ref{prop:genus-2} gives the crossing numbers of stable
genus~$2$ tropical curves. In this case, the tropical crossing number depends on the edge length.


The following proposition is also proved as Proposition~8.3 in~\cite{bjms},
where the curves satisfying the hypothesis are called sprawling.

\begin{prop}\label{prop:lollipop}
Let $\Gamma$ be a trivalent abstract tropical curve and let $v$ be a degree~$3$
vertex of~$\Gamma$ such that removing~$v$ disconnects $\Gamma$ into three
components~$A$, $B$, and~$C$. We consider each of these components to be a
tropical curve by including $v$ in each of them.

Suppose that $A$, $B$, and~$C$ each contain at least one trivalent vertex.
Then $\Gamma$ has a planar embedding if and only if $A$, $B$,
and~$C$ each contain a single trivalent vertex, in which case $\Gamma$ looks like
the curve shown in Figure~\ref{fig:windmill}.
\end{prop}

\begin{figure}
\includegraphics{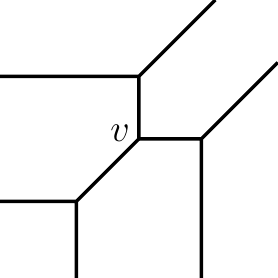}
\caption{A smooth tropical curve whose underlying graph structure is the
windmill graph.}\label{fig:windmill}
\end{figure}

\begin{proof}
Consider a planar embedding of $\Gamma$.  By making a change of coordinates, we
may assume that the outgoing edges from $v$ have directions $(1,0)$, $(0,1)$,
and $(-1,-1)$.  Let $w_A$, $w_B$, and $w_C$ be the first trivalent vertices in these
respective directions, which we've assumed to exist.  We will show that $v$,
$w_A$, $w_B$, and~$w_C$ are the only trivalent vertices.

The local model for a smooth curve implies that the outgoing directions at~$w_A$
are $(-1,0)$, $(-a, -1)$, and $(a+1, 1)$ for some $a \in \ZZ$. Similarly, the
outgoing directions at~$w_B$ are $(0, -1)$, $(-1, -b)$, and $(1, b+1)$ for some
$b \in \ZZ$ and at $w_C$ they are $(1, 1)$, $(c, c-1)$, and $(-c-1, -c)$ for
some $c \in \ZZ$.

Consider the component of the complement $\RR^2 \setminus \Gamma$ which lies
between $w_A$ and $w_B$. This component will be convex because the angle between
any two rays of the standard smooth model is less than $180^\circ$, even after
any change of coordinates.
Furthermore, because the only paths between
$w_A$ and $w_B$ pass through $v$, this component must be unbounded. On the other
hand, we have the edge from $w_A$ in the direction $(a+1, 1)$, so in particular,
with increasing $y$-coordinate, and we have an edge from $w_B$ with slope $b+1$,
so if $b+1 \leq 0$, then these two directions would eventually meet, and thus
they could not form edges of a convex, unbounded region. We conclude that $b
\geq 0$. Applying the same analysis to the region between $w_C$ and~$w_B$, we
get that $b \leq 0$, so $b=0$ and by symmetry $a=c = 0$ as well.

At this point, we know that, in a neighborhood of $v$, $w_A$, $w_B$, and~$w_C$,
the embedding of $\Gamma$ must
look like in Figure~\ref{fig:windmill}. However, the region between $w_A$
and~$w_B$ now has two edges parallel to $(1,1)$ in its boundary, and as we've
seen, this region must be unbounded, so these edges are also unbounded.
By symmetry, each of the vertices $w_A$, $w_B$, and~$w_C$ must have two
unbounded edges from it. Thus, there are no further trivalent vertices, which is
what we wanted to show.
\end{proof}

\begin{figure}
\includegraphics{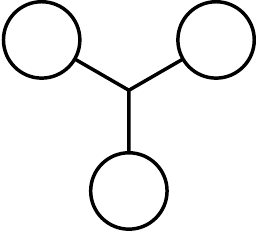}
\caption{The lollipop abstract tropical curve, which does not have a planar
embedding by Corollary~\ref{cor:lollipop-curve}.}
\label{fig:lollipop-curve}
\end{figure}

\begin{cor}\label{cor:lollipop-curve}
The ``lollipop curve'' shown in Figure \ref{fig:lollipop-curve} has crossing
number at least~$1$, for any lengths.
\end{cor}

Corollary~\ref{cor:lollipop-curve} is a strengthening of the last sentence of
Proposition~2.3
from~\cite{bitangents}. In that paper, they studied smooth plane quartics
and found examples of
all combinatorial types of genus~3 graphs except for the lollipop graph.
Corollary~\ref{cor:lollipop-curve} further shows that the lollipop graph does
not have a planar embedding, even if the curve is not required to be quartic.

\begin{figure}
\includegraphics{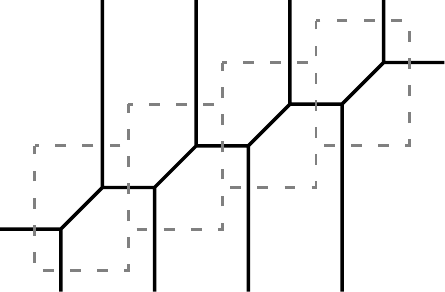}
\caption{Embedding of the caterpillar graph with $10$ leaves. Caterpillar graphs
with more leaves can also be embedded by adding more repetitions of the blocks
in the dashed lines.}
\label{fig:caterpillar}
\end{figure}

\begin{prop}\label{prop:tree}
Let $\Gamma$ be a genus~$0$ abstract tropical curve. Then $\Gamma$ has crossing
number~$0$ if and only if the underlying graph is a subdivision of the
caterpillar graph or the windmill graph.
\end{prop}

\begin{proof}
In the case of the caterpillar graph or the windmill graph, the embedding can be
constructed as in Figures~\ref{fig:caterpillar} and~\ref{fig:windmill}
respectively. Conversely, if the tree~$\Gamma$ does have a planar embedding,
then we can apply Proposition~\ref{prop:lollipop} at each trivalent vertex, and
there are two possibilities. If, for some vertex~$v$, all components of $\Gamma
\setminus \{v\}$ have only a single trivalent vertex, then $\Gamma$ is a
windmill graph. On the other hand, if for all trivalent vertices~$v$, at least
one of the components of $\Gamma \setminus \{v \}$ has no trivalent vertices, then $\Gamma$ is a
caterpillar graph.
\end{proof}

\begin{prop}\label{prop:sun}
Let $\Gamma$ be a curve whose underlying graph is a sun: a cycle with $n$ edges
attached to the
cycle. Then the crossing
number of~$\Gamma$ is at least $\lceil n/2 \rceil - 4$ if $n > 9$ and this bound
is sharp for some edge lengths.
\end{prop}

\begin{proof}
If we have an immersion of~$\Gamma$
with $k$ nodes, then the dual triangulation will be a polygon with $k+1$
interior lattice points by Proposition~\ref{prop:interior-points}. Since each ray of the sun graph will produce an
unbounded edge, the polygon has at least $n$ edges and so at least $n$ lattice
points on its boundary. Scott's inequality is a
bound on
the number $b$ of lattice points on the boundary of a polygon in terms of the
number in the interior lattice points~\cite{scott} (see
also~\cite{haase-schicho}). When $n > 9$, so that $b \geq n > 9$, Scott's
inequality has the form
$b \leq 2 (k+1) + 6$. Since $n \leq b$, we can then
solve for $k$ to get the relation $k \geq n/2 -4$, from which the desired
inequality follows because the crossing number must be an integer.

\begin{figure}
\includegraphics{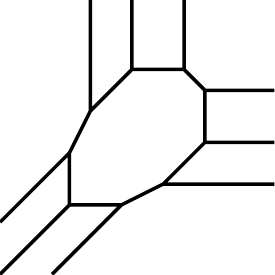}
\caption{Embedding of the sun curve with $9$ or fewer infinite edges.}
\label{fig:sun-9}
\end{figure}

\begin{figure}
\includegraphics{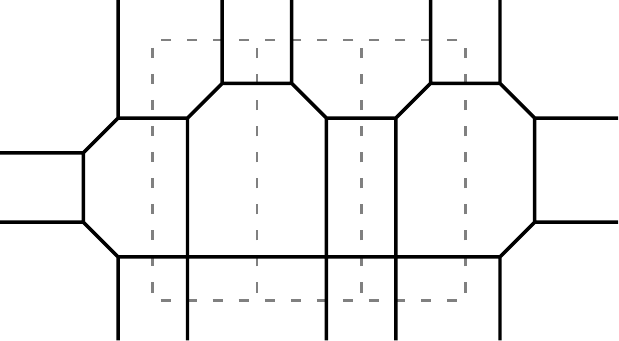}
\caption{Immersion of the sun curve with more than $9$ infinite edges. In the
example depicted,
the immersed tropical curve has $14$ infinite edges and thus $3$
crossings. The number of crossings can be varied by adding or removing blocks of
the form shown in the dotted boxes. Each such block adds two infinite edges and
one crossing.}
\label{fig:sun-general}
\end{figure}

To show that this bound can be sharp, we give examples of immersions of
tropical curves which
achieve them.
For $n\leq 9$, the embedding in Figure~\ref{fig:sun-9} justifies the requirement
of $n > 9$ from the proposition statement.
On the other hand, the pattern in Figure~\ref{fig:sun-general} with
$k$ of the dotted blocks will give an immersion with $k$~crossings and $2k + 8$
infinite edges. For $n > 9$, by taking $k = \lceil n / 2 \rceil - 4$, we have an
example which has $n$ or $n+1$ infinite edges and achieves the crossing number
bound from the proposition statement.
\end{proof}

\begin{prop}\label{prop:symmetric-theta}
Let $\Gamma$ be the theta graph as in the right of Figure~\ref{fig:theta-example}, with all edge lengths equal.
Then $\Gamma$ has crossing number~$1$.
\end{prop}

The first step in the proof of Proposition~\ref{prop:symmetric-theta} is the following
Lemma~\ref{lem:newton-polygon-genus-2}, which constrains the possible shapes of
the Newton polygon of an embedding of the theta graph, or of any other genus~$2$
graph. The possibilities in the conclusion of
Lemma~\ref{lem:newton-polygon-genus-2} are illustrated in
Figure~\ref{fig:genus-2-polygon}.

\begin{lem}\label{lem:newton-polygon-genus-2}
Let $\Delta$ be the Newton polygon dual to a smooth embedding of  an
abstract tropical curve of genus two. Then we can choose an affine change of coordinates such
that the interior points of~$\Delta$ are $(0,0)$ and $(1,0)$ and such that
either of the following inequalities hold:
\begin{enumerate}
\item The points of $\Delta$ are bounded by $-1 \leq y \leq 1$ and $x \leq 2$.
\item The points of $\Delta$ are bounded by $-1 \leq y \leq 1$ and $x \geq -1$.
\end{enumerate}
Moreover, the transformation changing from coordinates satisfying (1) to those
satisfying (2) acts as the identity on the $y=0$ line.
\end{lem}

\begin{figure}
\includegraphics{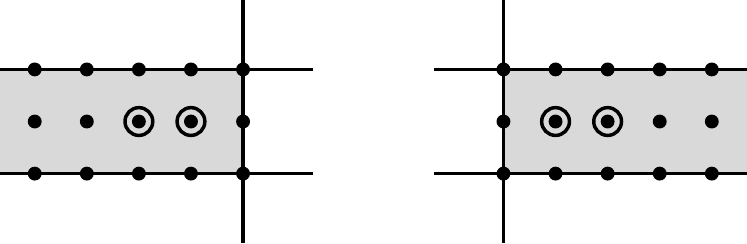}
\caption{Lemma~\ref{lem:newton-polygon-genus-2} says that the Newton polygon can
be constrained to one of these two shaded regions. In either case, the interior
points of the Newton polygon are the circled dots, which have coordinates $(0,
0)$ and $(1,0)$.}\label{fig:genus-2-polygon}
\end{figure}

\begin{proof}
Since the embedding, by
definition, has no nodes, Proposition~\ref{prop:interior-points} implies that
$\Delta$ has exactly two integral points in its interior. After an affine change
of coordinates which is common to both (1) and (2) from the statement, we can
assume that these two interior points are $(0, 0)$ and $(1,0)$. Now suppose that
there is a vertex of the polygon $(a,b)$ with $b
> 1$. We claim that the integral point $v = (\lceil a/b \rceil, 1)$ would
also be in the interior of~$\Delta$, which would be a contradiction. To see this
claim, let $r = \lceil a/b\rceil - a/b$, and note that $0 \leq r \leq 1 - 1/b$.
We can then write:
\begin{equation*}
v = \left(1 - r - {\textstyle\frac{1}{b}}\right) (0, 0)
+ r (1, 0)
+ \left({\textstyle\frac{1}{b}}\right) (a,b),
\end{equation*}
which is a convex linear combination by the previously noted inequalities. Since
$(0, 0)$, $(1,0)$, and $(a,b)$ are all in~$\Delta$, so is $v$. Moreover, the
former two points are in the interior of~$\Delta$, and $v$ is not equal to
$(a,b)$, so $v$ is in the interior of~$\Delta$.
Thus, we've proved the desired claim, and so there can be
no interior lattice point $(a,b)$ with $b > 1$. By symmetry, there are no vertices $(a,b)$
in~$\Delta$ with $b < -1$ either, and so we've proved the first half of either
set of bounds.

We now turn to finding bounds on the horizontal extent of~$\Delta$.
Suppose that $(a,1)$ is the rightmost vertex of~$\Delta$ on the $y = 1$ line. We
now change coordinates by sending $(x,y)$ to $(x - (a-2)y, y)$, after which
$(2,1)$ is the rightmost such point. We now claim that all points in~$\Delta$
are bounded by $x \leq 2$. On the $y=1$ line, this is by our change of
coordinates, and if there exists a vertex $(a, b)$ with $a > 2$ and either $b =
0$ or $b=-1$, then the vertex $(2,0)$ would be in the interior of~$\Delta$,
which is a contradiction.
Therefore, we've found coordinates satisfying the bounds in~(1).

To find the second set of coordinates, we start with the leftmost vertex on the
$y=1$ and apply an analogous change of coordinates, which is the identity on the
$y=0$ line and after which $\Delta$ is bounded by $x \geq -1$, as desired.
\end{proof}

\begin{proof}[Proof of Proposition~\ref{prop:symmetric-theta}]
The immersion from Example~\ref{ex:theta-immersion} shows that the crossing
number is at most~$1$, and so it remains to show that there is no planar
embedding of~$\Gamma$. 

We assume for the sake of contradiction that we have a planar embedding~$\iota$.
We consider the dual triangulation of the Newton polygon~$\Delta$, for which we
first assume we have coordinates for which $\Delta$ is bounded as in
Lemma~\ref{lem:newton-polygon-genus-2}(1). Since $\Gamma$ is a theta graph,
there must be an edge of the embedding separating the two bounded regions.
Dually, the bounded regions correspond to the points $(0,0)$ and $(1,0)$, so
there must be an edge of the triangulation between them, which then corresponds
to a vertical edge in $\iota(\Gamma)$. The triangles above and below the edge joining $(0,0)$ and $(1,0)$
correspond to the two trivalent vertices of~$\Gamma$. We label the edges of
$\Gamma$ as $e_1$, $e_2$, and~$e_3$, such that $\iota(e_2)$ is the vertical
edge, and the regions to the left and right of this edge are bounded by
$\iota(e_1 \cup e_2)$ and $\iota(e_2 \cup e_3)$ respectively.

Now we consider the subset $e_3'$ of~$e_3$ consisting of points $p$ of~$e_3$
such that $\iota(p) + (\epsilon, 0)$ is in an unbounded region or unbounded ray for all
sufficiently small~$\epsilon$. Equivalently, $\iota(e_3')$ is the union of the
segments of $\iota(e_3)$ whose dual in the triangulation are edges connecting
$(1,0)$ and a point with first coordinate greater than~$1$.
By the inequalities of Lemma~\ref{lem:newton-polygon-genus-2}(1), the
only possibilities for the second point are $(2,1)$, $(2,0)$, and $(2,-1)$, which
correspond to edges in $\iota(e_3')$ parallel to the vectors $(1,-1)$,
$(0,-1)$ and $(-1,-1)$, respectively. All of these vectors have $-1$ in the
second coordinate, so the total length of the segment~$e_3'$
equals the height of~$\iota(e_3')$. Since the length of~$e_3$ equals the length
of $e_2$, which equals the height of the vertical segment $\iota(e_2)$, we know that  $e_3'$
must consist of all of~$e_3$.
As a consequence, the only possible endpoints of edges of the triangulation containing $(1,0)$ are
$(0,0)$, $(2,1)$, $(2,0)$, and $(2,-1)$. Thus, the triangles
above and below the edge from $(0,0)$ to~$(1,0)$ in the triangulation must
contain the vertices $(2,1)$ and $(2,-1)$ respectively. The midpoint of these
two vertices will be $(2,0)$ and note that this midpoint is preserved under
linear changes of coordinates which also preserve the $y=0$ line.

Second, we consider coordinates such that $\Delta$ is bounded as in
Lemma~\ref{lem:newton-polygon-genus-2}(2). By symmetry, the same argument
applied to~$e_1$ shows that, in these coordinates, the triangles above and below
the edge between $(0,0)$ and $(1,0)$ have their third vertices at $(-1, 1)$ and $(-1, -1)$, respectively.
The
midpoint of these two vertices is $(-1,0)$, which would remain true when
changing to the coordinates as in Lemma~\ref{lem:newton-polygon-genus-2}(1).
Therefore, we have a contradiction with the previous paragraph, so there is no
embedding of the graph~$\Gamma$, so its crossing number is~$1$.
\end{proof}

\begin{prop}\label{prop:genus-2}
Let $\Gamma$ be a stable tropical curve of genus two. Then $\Gamma$ has
tropical crossing number $0$ unless $\Gamma$ is the theta graph in
Figure~\ref{fig:theta-example}
with all edge lengths equal, in which case it has crossing number~$1$.
\end{prop}

\begin{figure}
\includegraphics{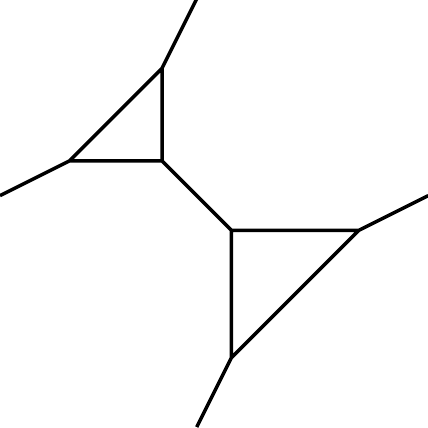}
\caption{Embedding of the barbell graph.}\label{fig:barbell}
\end{figure}

\begin{figure}
\includegraphics{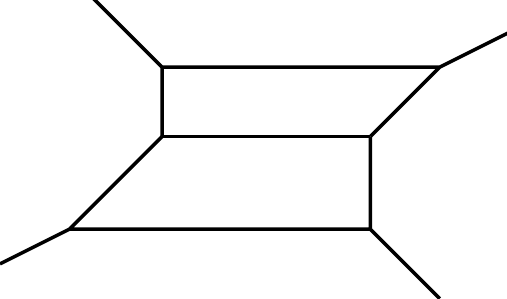}
\caption{Embedding of the theta graph when
$a < b \leq c$. In this figure, the middle edge has
length~$a$.}\label{fig:theta-unequal}
\end{figure}

\begin{figure}
\includegraphics{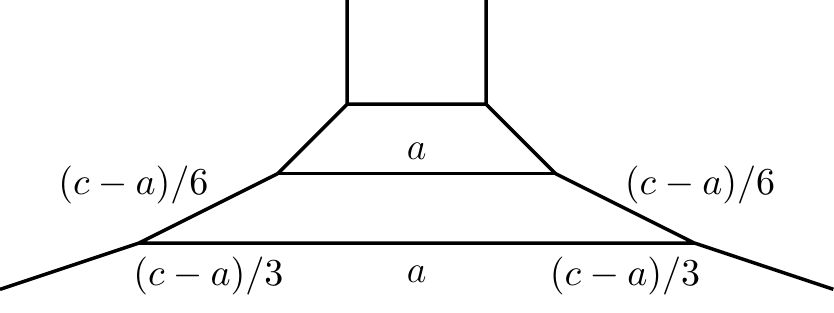}
\caption{Embedding of the theta graph with edge lengths~$a$, $b$, and~$c$ when
$a = b<c$.}\label{fig:theta-equal}
\end{figure}

\begin{proof}
There are two combinatorial types of trivalent graphs of genus~$2$. In the case
of the barbell graph, we can take the planar embedding shown in
Figure~\ref{fig:barbell} for
all possible edge lengths. For the theta graph in Figure~\ref{fig:theta-example}, there are three
possibilities depending on the edge lengths~$a$, $b$, and~$c$.
By symmetry, we can assume that $a
\leq b \leq c$. If we have a strict inequality $a < b$ or if $a = b < c$, then
we can use the embeddings in Figures~\ref{fig:theta-unequal}
and~\ref{fig:theta-equal} respectively. Otherwise, all of the edge lengths are
equal, and the crossing number is~$1$ by Proposition~\ref{prop:symmetric-theta}.
\end{proof}

\begin{figure}
\includegraphics{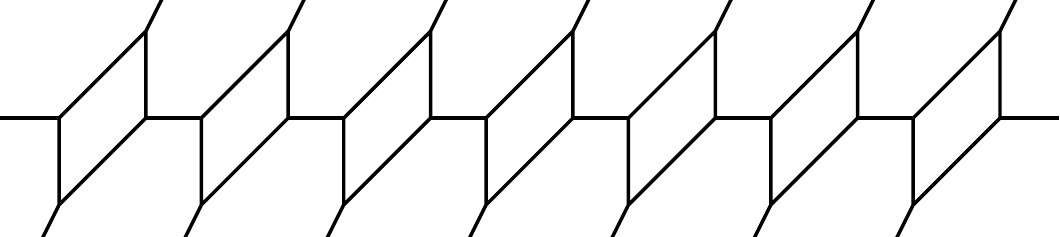}
\caption{Embedding without crossings of a chain of $7$ loops when both halves of
each loop have the same length. The same pattern can be continued to give an
embedding if the graph is extended to have arbitrarily many
loops.}\label{fig:chain-loops}
\end{figure}

\begin{rmk}\label{rmk:semicontinuity}
The example of the theta graph in Proposition~\ref{prop:genus-2} shows that
the crossing number can increase for specializations of the metric parameters,
i.e.\ the crossing number function is not lower semi-continuous. Moreover, it is
not upper semi-continuous because the crossing
number can jump down in specializations as well. For example, the chain of $g \geq 15$ loops
with generic edge lengths has positive crossing number by
Corollary~\ref{cor:chain-loops}, but if all edge lengths are equal then it is
planar by the embedding shown in Figure~\ref{fig:chain-loops}. Also, a consequence of the proof of Proposition~\ref{prop:sun} is that
any embedding of a sun curve with $9$ infinite edges is equivalent to that in
Figure~\ref{fig:sun-9}, up to change of coordinates in $\GL_2(\mathbb{Z})$. However, such an embedding
implies non-trivial conditions among the lengths of the edges of the cycle.
\end{rmk}

For curves of higher genus, we refer to~\cite[Sec.~5--8]{bjms}, where curves of
crossing number $0$ are characterized using computational techniques. In
particular, they show that it is feasible to enumerate the possible Newton
polygons of low genus curves and from these compute the defining inequalities of
the tropical curves which can then arise.

\section{Algebraic realizations of tropical curves}\label{sec:realizations}

In this section, we consider applications of our results to realizations of
tropical curves by algebraic curves. We recall that for any curve~$C \subset
\Gm^2$ over a field~$K$ with valuation, the tropicalization of~$C$ is a union of
finitely many edges in~$\RR^2$.
One characterization of the tropicalization $C \subset \Gm^2$ is as the
projection of the Berkovich analytification~$C^\an$ using the valuations of the
coordinate functions of~$\Gm^2$. We are interested in cases when this map
preserves the skeleton of the analytification in the following sense:

\begin{defn-boxless}\label{def:faithful}
Let $C \subset \Gm^2$ be a nodal curve over a field~$K$ with non-trivial
valuation and let $\widetilde C$ be the normalization of~$C$. We say that that
$C$ has \defi{totally faithful tropicalization} if there exists a finite set of
points $p_1, \ldots, p_n \in \RR^2$ such that:
\begin{enumerate}
\item The skeleton of~$\widetilde C$ maps isometrically onto the tropicalization
of~$C$, except at $p_1, \ldots, p_n$.
\item At each point $p_i$, the tropicalization of~$C$ is a node, as in
Definition~\ref{def:nodal}, and the number of nodes of~$C$ tropicalizing
to~$p_i$ is equal to the multiplicity of the node at~$p_i$, as
in~(\ref{eq:node-multiplicity}). \qed
\end{enumerate}
\end{defn-boxless}

Definition~\ref{def:faithful} is a generalization to nodal curves of the
definition given in~\cite{log-deformation}. Baker, Payne, and Rabinoff have
shown that any algebraic curve~$C$ has a faithful tropicalization in the sense
that there's an open subset~$C' \subset C$ and an embedding $C' \rightarrow
\Gm^n$ such that the map from the analytification~$(C')^\an$ is an isometry on
the skeleton of~$C$, but this is not necessarily a totally faithful
tropicalization because the skeleton of $C'$ is larger than that of~$C$ when $C'
\subsetneq C$. Theorem~\ref{thm:faithful-immersion} constructs totally faithful
tropicalizations, but working in reverse, beginning with an abstract tropical
curve and constructing the algebraic curve.

\begin{figure}
\includegraphics{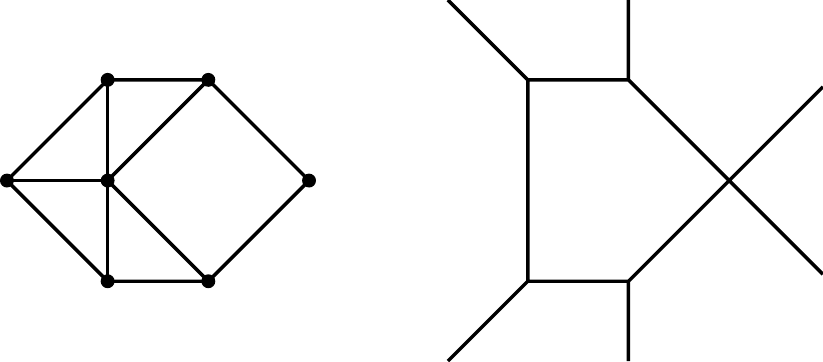}
\caption{On the right is a nodal tropical curve which is the totally faithful
tropicalization of the nodal rational curve given in Example~\ref{ex:faithful}.
On the left is the dual subdivision of the Newton polygon for this curve.}
\label{fig:faithful}
\end{figure}

\begin{ex}\label{ex:faithful}
Let $K = \mathbb C[[t]]$ and let $C$ be the curve in $\Gm^2$ defined by the
equation:
\begin{multline*}
(1+t)^2 t^2 x^3 y + t x^2 y^2 + (3t^4+3t^3+2t+t+1)x^2y  - xy^2 \\
+ tx^2 + (3t^3+t-2)txy  - x + (1-t)^2(1+t+t^2)y.
\end{multline*}
The dual subdivision of the Newton polygon and the corresponding tropical curve
are shown in Figure~\ref{fig:faithful}. The tropical curve is nodal with a
single
node of multiplicity~$2$ at the point $(0,0)$. While a generic curve with this Newton polygon would
have genus~$2$, one can check that $C$ has nodal singularities at the two
points:
\begin{equation*}
\big((1-t)/(1+t), 1\big) \quad\mbox{and}\quad \big((-1-t)/(1+t),-1\big),
\end{equation*}
and therefore, $C$ is rational. Both of these singularities have coordinates
with valuation~0 and therefore lie above the multiplicity~$2$ node in the
tropicalization.
Therefore, the nodal curve~$C$ has totally faithful tropicalization.
\end{ex}

\begin{proof}[Proof of Theorem~\ref{thm:faithful-immersion}]
By Theorem~\ref{thm:immersion-embedding}, $\Gamma$ admits a planar immersion
$\iota\colon \Gamma \rightarrow \RR^2$ and by
Remark~\ref{rmk:immersion-embedding-extras}, we can assume that the vertices in
this immersion have rational coordinates. Then, it will be sufficient to show
that this nodal plane curve is the tropicalization of some algebraic curve, for
which we use Mikhalkin's correspondence theorem~\cite{mikhalkin}, and
more specifically the algebraic version due to Shustin~\cite[Thm.~3]{shustin}.
While the cited theorem only asserts that a certain count of nodal algebraic
curves equals a weighted count of the corresponding tropical curves, the proof
works by constructing at least one algebraic curve for each nodal tropical
curve.
In particular,
\cite[Sec.~3.7]{shustin} shows that for any nodal tropical
curve~$\iota(\Gamma)$, it is possible to find certain auxiliary data, denoted by
$S$, $F$, and~$R$ in that paper. Then, \cite[Lem.~3.12]{shustin} states that from
$\iota(\Gamma)$, together with this auxiliary data,
we can find a nodal algebraic curve tropicalizing to~$\iota(\Gamma)$.
\end{proof}

\bibliographystyle{alpha}
\bibliography{tropical_crossing}

\newcommand{\dummyletter}[1]{#1}
\begin{thebibliography}{{\dummyletter{B}L}MPR14}

\bibitem[ABBR15]{abbr}
Omid Amini, Matthew Baker, Erwan Brugall{\'e}, and Joseph Rabinoff.
\newblock Lifting harmonic morphisms {I}: metrized complexes and {Berkovich}
  skeleta.
\newblock \textit{Res. Math. Sci}, to appear, \arxiv{1303.4812}, 2015.

\bibitem[AK06]{ardila-klivans}
Federico Ardila and Carly Klivans.
\newblock The {Bergman} complex of a matroid and phylogenetic trees.
\newblock {\em J. Combin. Theory Ser. B}, 96(1):38--49, 2006.

\bibitem[Bak08]{baker}
Matthew Baker.
\newblock Specialization of linear systems from curves to graphs.
\newblock {\em Algebra Number Theory}, 2(6):613--653, 2008.

\bibitem[BJMS15]{bjms}
Sarah Brodsky, Michael Joswig, Ralph Morrison, and Bernd Sturmfels.
\newblock Moduli of tropical plane curves.
\newblock \textit{Res. Math. Sci.}, to appear, \arxiv{1409.4395}, 2015.

\bibitem[{\dummyletter{B}L}MPR14]{bitangents}
Matt Baker{\global\def\dummyletter#1{},}~Yoav {\dummyletter{B}L}en, Ralph
  Morrison, Nathan Pflueger, and Qingchun Ren.
\newblock Bitangents of tropical plane quartic curves.
\newblock preprint, \arxiv{1404.7568}, 2014.

\bibitem[BPR11]{bpr}
Matthew Baker, Sam Payne, and Joseph Rabinoff.
\newblock Nonarchimedean geometry, tropicalization, and metrics on curves.
\newblock preprint, \arxiv{1104.0320}, 2011.

\bibitem[CC12]{castryck-cools}
Wouter Castryck and Filip Cools.
\newblock Newton polygons and curve gonalities.
\newblock {\em J. Algebraic Combin.}, 35(3):345--366, 2012.

\bibitem[CDPR12]{cdpr}
Filip Cools, Jan Draisma, Sam Payne, and Elina Robeva.
\newblock A tropical proof of the {Brill-Noether} theorem.
\newblock {\em Adv. Math.}, 230(2):759--776, 2012.

\bibitem[CFPU14]{log-deformation}
Man-Wai Cheung, Lorenzo Fantini, Jennifer Park, and Martin Ulirsch.
\newblock Faithful realizability of tropical curves.
\newblock preprint, \arxiv{1410.4152}, 2014.

\bibitem[FTM74]{toth-makai}
L{\'a}szl{\'o} Fejes-T{\'o}th and Endre {Makai Jr.}
\newblock On the thinnest non-separable lattice of convex plates.
\newblock {\em Studia Sci. Math. Hungar.}, 9:191--193, 1974.

\bibitem[GK08]{gathmann-kerber}
Andreas Gathmann and Michael Kerber.
\newblock A {Riemann-Roch} theorem in tropical geometry.
\newblock {\em Math. Zeit.}, 259(1):217--230, 2008.

\bibitem[GRW14]{grw}
Walter Gubler, Joseph Rabinoff, and Annette Werner.
\newblock Skeletons and tropicalizations.
\newblock preprint, \arxiv{1404.7044}, 2014.

\bibitem[Har77]{hartshorne}
Robin Hartshorne.
\newblock {\em Algebraic Geometry}, volume~52 of {\em Graduate Texts in
  Mathematics}.
\newblock Springer, 1977.

\bibitem[HS09]{haase-schicho}
Christian Haase and Josef Schicho.
\newblock Lattice polygons and the number {$2i+7$}.
\newblock {\em Amer. Math. Monthly}, 116(2):151--165, 2009.

\bibitem[JP14]{jensen-payne}
David Jensen and Sam Payne.
\newblock Tropical independence {I}: Shapes of divisors and a proof of the
  {Giesker--Petri} theorem.
\newblock {\em Algebra Number Theory}, 8(9), 2014.

\bibitem[Mik05]{mikhalkin}
Grigory Mikhalkin.
\newblock Enumerative tropical algebraic geometry in {$\mathbb R^2$}.
\newblock {\em J. Amer. Math. Soc.}, 18:313--377, 2005.

\bibitem[MS15]{maclagan-sturmfels}
Diane Maclagan and Bernd Sturmfels.
\newblock {\em Introduction to Tropical Geometry}.
\newblock Graduate Studies in Mathematics. Amer. Math. Soc., 2015.

\bibitem[MZ08]{mikhalkin-zharkov08}
Grigory Mikhalkin and Ilia Zharkov.
\newblock Tropical curves, their {Jacobians} and theta functions.
\newblock In {\em Curves and Abelian varieties}, volume 465 of {\em Contemp.
  Math.}, pages 203--230. Amer. Math. Soc., 2008.

\bibitem[MZ14]{mikhalkin-zharkov}
Grigory Mikhalkin and Ilia Zharkov.
\newblock Tropical eigenwave and intermediate {Jacobians}.
\newblock In {\em Homological Mirror Symmetry and Tropical Geometry}, volume~15
  of {\em Lecture Notes of the Unione Matematica Italiana}, pages 309--349.
  Springer, 2014.

\bibitem[RGST05]{rgst}
J{\"u}rgen Richter-Gebert, Bernd Sturmfels, and Thorsten Theobald.
\newblock First steps in tropical geometry.
\newblock In {\em Idempotent mathematics and mathematical physics}, volume 377
  of {\em Contemp. Math.} Amer. Math. Soc., 2005.

\bibitem[RS09]{richter-salazar}
R.~Bruce Richter and G.~Salazar.
\newblock Crossing numbers.
\newblock In {\em Topics in Topological Graph Theory}, volume 128 of {\em
  Encyclopedia Math. Appl.}, pages 133--150. Cambridge Univ. Press, 2009.

\bibitem[Sco76]{scott}
Paul~R. Scott.
\newblock On convex lattice polygons.
\newblock {\em Bull. Austral. Math. Soc.}, 15(3):395--399, 1976.

\bibitem[Shu05]{shustin}
Eugenii Shustin.
\newblock A tropical approach to enumerative geometry.
\newblock {\em Algebra i analiz}, 17(2):170--214, 2005.
\newblock (English translation: {\textit{St. Petersburg Math. J.}},
  17(2):343--375, 2006).

\bibitem[Smi15]{smith}
Geoffrey Smith.
\newblock {Brill-Noether} theory of curves on toric surfaces.
\newblock {\em J. Pure Appl. Algebra}, 219(7):2629--2636, 2015.

\end{thebibliography}

\end{document}